\documentclass[11pt,reqno]{amsart}

\usepackage{setspace}
\usepackage[T1]{fontenc}
\usepackage{amssymb,amscd,amsthm,latexsym}
\usepackage{graphicx}
\usepackage[british]{babel}
\usepackage[all]{xy}
\usepackage{color}
\usepackage{mathabx}
\usepackage{calrsfs}
\usepackage{booktabs}
\usepackage{chngcntr}

\theoremstyle{plain}
\newtheorem{theorem}[subsection]{Theorem}

\newtheorem{proposition}[subsection]{Proposition}
\newtheorem{corollary}[subsection]{Corollary}

\theoremstyle{definition}
\newtheorem{definition}[subsection]{Definition}

\newtheorem{remark}[subsection]{Remark}

\renewcommand{\le}{\leqslant}

\newcommand{\pb}[1][dr]{\save*!/#1-1.5pc/#1:(-1,1)@^{|-}\restore}
\newcommand{\pbdown}[1][d]{\save*!/#1-1.5pc/#1:(-1,1)@^{|-}\restore}

\newcommand{\Grp}{\mathsf{Grp}}

\newcommand{\F}{\mathsf{F}}

\newcommand{\op}{\mathsf{op}}

\newcommand{\Eq}{\mathsf{Eq}}

\newcommand{\D}{\mathbb{D}}
\newcommand{\C}{\mathbb{C}}
\newcommand{\V}{\mathbb{V}}

\newcommand{\Sub}{\mathsf{Sub}}
\newcommand{\eot}[1]{\ensuremath{\in_{#1}}}

\newcommand{\mono}{\ensuremath{\rightarrowtail}}
\newcommand{\repi}{\ensuremath{\twoheadrightarrow}}

\newdir{ >}{
 @{}*!/-10pt/@{>} }

\begin{document}

\title{Categorical aspects of congruence distributivity}

\author{Michael Hoefnagel}
\address{Mathematics Division, Department of Mathematical Sciences, Stellenbosch University, Pri\-vate Bag X1 Matieland 7602, South Africa and National Institute for Theoretical and Computational Sciences (NITheCS), South Africa}
\email{mhoefnagel@sun.ac.za}

\author{Diana Rodelo}
\address{Department of Mathematics, University of the Algarve, 8005-139 Faro, Portugal and CMUC, Department of Mathematics, University of Coimbra, 3001-143 Coimbra, Portugal}
\thanks{The second author acknowledges partial financial support by the Centre for Mathematics of the University of Coimbra (funded by the Portuguese Government through FCT/MCTES, DOI 10.54499/UIDB/00324/2020).}
\email{drodelo@ualg.pt}



\begin{abstract}
We study a categorical condition on relations, which is a categorical formulation of J\'onsson's characterisation of congruence distributive varieties. Categories satisfying these conditions need not be varieties; for instance, the dual of the categories of topological spaces, ordered sets, $G$-sets, and the dual of any (pre)topos all provide us with examples. 
\end{abstract}{}

\maketitle

\section{Introduction} \label{section: introduction}
A variety of universal algebras $\V$ is \emph{congruence distributive}~\cite{Jonsson1967} if the congruence lattice on any algebra in $\V$ is a distributive lattice. These varieties occupy an extremal situation in universal algebra created by the modular commutator, namely, when it is as large as possible. On the other end, when it is as small as possible, we have the abelian varieties --- a topic which categorical algebra is well suited for. However, there has been comparably less attention given to congruence distributivity, from the categorical point of view. One possible reason for this, is that there are several categorical conditions arising from equivalent characterisations of congruence distributivity for varieties, each with their advantages and disadvantages. Moreover, these categorical conditions are not necessarily equivalent, as are their varietal counterparts. 

Congruences in a variety of universal algebras are just internal equivalence relations in the variety, so the most direct categorical distributivity condition is the notion of an equivalence distributive category given in \cite{GranRodeloNguefeu2020}: a finitely complete category $\C$ is called \emph{equivalence distributive} when the meet semilattice $\Eq(A)$ of equivalence relations on $A$ is a distributive lattice: for any $R,S,T\in \Eq(A)$, we have
\begin{equation}
\label{distributivity}
    R\wedge (S\vee T) = (R\wedge S)\vee (R\wedge T).
\end{equation}
However, this notion has several drawbacks. For one thing, its definition relies on the existence of suprema of equivalence relations --- a condition not guaranteed by familiar categorical contexts such as finite completeness, regularity~\cite{BarrGrilletOsdol1971}, or even Barr-exactness~\cite{BarrGrilletOsdol1971}. Moreover, the definition of equivalence distributive category excludes categories such the variety of (distributive) lattices equipped with an term of countable arity (see \cite{GJanelidze2016}). However, this category is categorically quite similar to the category of lattices --- the prototypical example of a congruence distributive variety.  

A condition which does not involve such an inductive requirement as the existence of suprema of equivalence relations, is a categorical version of the so called \emph{trapezoid lemma}~\cite{Chajda2003}. This condition captures (and is, thus, equivalent to) congruence distributivity in much the same way as H.~P.~Gumm's shifting lemma captures congruence modularity~\cite{Gumm1983}. A variety of universal algebras $\V$ satisfies the trapezoid lemma if for every three congruences $R,S,T$ on any algebra $A$ in $\V$ the following diagrammatic condition holds: 
\begin{equation}\label{trapezoid lemma}
R\wedge S\leqslant T\quad\quad\Rightarrow\quad\quad \vcenter{\xymatrix@=25pt{ 
u \ar@{-}[d]_S\ar@{-}[r]^T & v \\
x \ar@{..}@/_1.5pc/[r]_T \ar@{-}[r]_R & y. \ar@{-}[u]_S }}
\end{equation}
More precisely, if $R\wedge S\le T$ and elements $x,y,u,v$ of $A$ are such that $xSu$, $uTv$, $ySv$, $xRy$, then we can deduce that $xTy$. This condition on equivalence relations can be formulated internally to any category using standard techniques involving the Yoneda embedding, so that we can discuss categories satisfying the trapezoid lemma (see Definition~\ref{definition: trapezoid lemma}). Then a variety satisfies this categorical version of the trapezoid lemma if and only if it satisfies the trapezoid lemma in the above sense. It is easy to prove that any regular equivalence distributive category satisfies the trapezoid lemma (Proposition~\ref{equiv distributive => trapezoid lemma}). We do not know if the converse is true or not (supposing we are in a categorical context where suprema of equivalence relations exist). The main disadvantage in the trapezoid lemma approach is finding convenient equivalence relations $R, S, T$, with $R\wedge S\le T$, to which we can actually apply the property of the diagram in \eqref{trapezoid lemma}. 

In this paper, we introduce and study another possible categorical generalisation of congruence distributive varieties. This generalisation is motivated by B. J\'onsson's  characterisation of congruence distributive varieties \cite{Jonsson1967}, which states that a variety of universal algebras $\V$ is congruence distributive if and only if there exists an $n \geqslant 1$ and ternary  terms $p_1,\dots, p_n$ in the theory of $\V$  such that the equations 
\[
\begin{array}{ll}
    p_i(x,y,x)=x, & \text{for all $i$}, \\ 
    p_1(x,x,y)=x, \\
    p_i(y,x,x)=p_{i+1}(y,x,x), & \text{for $i$ odd}, \\
    p_i(x,x,y)=p_{i+1}(x,x,y), & \text{for $i$ even}, \\
    p_{n}(x,x,y)=y, & \text{if $n$ even}, \\ 
    p_{n}(y,x,x)=x, & \text{if $n$ odd}. \\
\end{array}
\]
hold. This condition can not be directly formulated in a categorical way, since it is a Mal'tsev condition, and the basic language of categories does not deal with terms. However, the existence of the $n$ \emph{J\'onsson terms} above in a variety $\V$ is equivalent to the requirement that for any three congruences $R,S,T$ on any algebra in the variety, we have  

\begin{equation}\label{Jonsson main pp}
R\wedge (S\circ T)\leqslant \underbrace{(R\wedge S) \circ (R\wedge T) \circ (R\wedge S) \circ \cdots}_{(n + 1) \text{ times}}.
\end{equation}
This property can be formulated categorically with respect to equivalence relations, and in the context of a regular category. A regular category $\C$ is called a \emph{J\'onsson category of order $n$} when any three equivalence relations $R, S, T$ on any object in $\C$ satisfies \eqref{Jonsson main pp}. The advantages of this generalisation is that it does not require the existence of suprema of equivalence relations and the property on equivalence relations \eqref{Jonsson main pp} is easier to check than that of the trapezoid lemma.

The aim of this paper is to show that several properties of congruence distributive varieties extend to properties of J\'onsson categories of some order $n$. We will also compare other categorical properties which are associated to congruence distributive varieties, such as the notion of a majority category \cite{Hoefnagel2019a, Hoefnagel2020a}, the notion of a locally anticommutative category in the sense of \cite{Hoefnagel2021}, as well as a categorical counterpart of the trapezoid lemma in the sense of \cite{Duda2000} (see also \cite{Chajda2003}). 

\section{Preliminaries} \label{section: preliminaries}
In this section we recall the notions of subobject, internal relation, as well as the notion of regular category and their properties. For the reader who is already familiar with these notions, we recommend taking note of the notation and conventions described in \ref{subsubsection: notation} below and skipping to Section~\ref{section: trapezoid lemma}. 

\subsection{Subobjects and relations in categories} \label{section: subobjects}
If $m:M_0 \mono X$ and $n:N_0 \mono X$ are monomorphisms in a category $\C$, then we write $m \leqslant n$ if $m$ factors through $n$, i.e., if there exists $\phi:M_0 \rightarrow N_0$ such that $n\phi = m$. This defines a preorder on $\mathsf{Mono}(X)$, the class of all monomorphisms in $\C$ with codomain $X$. The posetal reflection of $\mathsf{Mono}(X)$ is called the \emph{poset of subobjects} of $X$, and is denoted by $\Sub(X)$. Explicitly, a \emph{subobject} $S \in \Sub(X)$ is an equivalence class of monomorphisms with codomain $X$, where two monomorphisms $n,m \in \mathsf{Mono}(X)$ are equivalent if and only if $n \leqslant m$ and $m \leqslant n$. If $s: S_0 \mono X$ is a member of $S$, then we will say that $S$ is the \textit{subobject represented} by $s$ in what follows.  

In any category $\C$ the pullback of a monomorphism along any morphism is again a monomorphism, which is to say that if the diagram:
\[
\xymatrix{
	\bullet \ar@{ >->}[d]_n \ar[r] \pb & \bullet \ar@{ >->}[d]^m \\
	\bullet \ar[r]_f & \bullet 
}
\]
is a pullback diagram in $\C$, and $m$ is a monomorphism, then so is $n$. Given that $\C$ has pullbacks of monomorphisms along monomorphisms and $A,B \in \Sub(X)$ are any subobjects represented by $a:A_0 \mono X$ and $b:B_0 \mono X$ respectively, then we write $A \wedge B$ for the subobject of $X$ represented by the diagonal monomorphism in any pullback 
\[
\xymatrix{
	(A\wedge B)_0\ar[d]_-{p_2} \ar[r]^-{p_1} \ar@{ >->}[dr]& A_0 \ar@{ >->}[d]^{a} \\
	B_0 \ar@{ >->}[r]_{b} & X
}
\]

\subsubsection{Notation and conventions}\label{subsubsection: notation}
The notional conventions given here will be used without mention throughout the rest of the text. Let $A$ be a subobject of an object $X$ and consider a morphism $x\colon S\to X$ in a category $\C$, then:
\begin{itemize}
 \item if $x$ factors through one representative of $A$, then it factors through all representatives of $A$;
 \item we write $x \eot{S} A$ if $x$ factors through a representative of $A$
	\[
	\xymatrix{
		& X \\
		S \ar@{..>}[r] \ar[ru]^{x} & A_0; \ar@{ >->}[u]_a
	}
	\]
 \item if $\C$ has pullbacks and $B$ is a subobject of $X$, then we have $x \in_S A \wedge B$ if and only if $x \in_S A$ and $x \in_S B$;
 \item if $e: E \repi S$ is a regular epimorphism in $\C$, then $xe \eot{E} A$ if and only if $x \eot{S} A$. In fact, this holds if $e$ is a strong epimorphism, as it relies only on the diagonal fill in property of strong epimorphisms.
\end{itemize}

\begin{definition}
	Let $\C$ be a category with finite products, then an \emph{(internal) $n$-ary relation} $R$ between objects $X_1,X_2,...,X_n$ in $\C$ is simply a subobject of $X_1 \times X_2 \times \cdots \times X_n$. 
\end{definition}

We will adopt a standard abuse of notation, namely, we will use the same $R$ to denote the domain of the monomorphism which represents the relation, i.e., a relation $R$ as above is denoted by $R\mono X_1\times \cdots \times X_n$. Also, if $(r_1,r_2)\colon R\mono X\times Y$ is a relation from $X$ to $Y$, the \emph{opposite} of $R$, denoted $R^\op$, is the relation from $Y$ to $X$ given by $(r_2,r_1)\colon R\mono Y\times X$. For a binary relation $R$, and morphisms $x,y$ in $\C$, we will write $x R y$ if and only if $\mathsf{dom}(x) = S = \mathsf{dom}(y)$ and $(x,y) \in_S R$.

The notation introduced above allows for a simple translation of the familiar properties of binary relations. Given a relation $R$ on an object $A$ (i.e. a relation from $A$ to $A$), we say that $R$ is:
\begin{itemize}
    \item \emph{reflexive} if for any morphism $x:S \to A$ in $\C$ we have $x R x$. This is equivalent to the requirement that $\Delta_A \leqslant R$, where $\Delta_A$ is the diagonal relation on $A$, i.e., the identity relation represented by the monomorphism $(1_A, 1_A):A \mono A \times A$;
    \item \emph{symmetric} when $R^\op\leqslant R$;
    \item \emph{transitive} when for any morphisms $x,y,z$ if $xRy$ and $y R z$ then $x R z$;
    \item an \emph{equivalence relation} when it is reflexive, symmetric and transitive;
    \item an \emph{effective equivalence relation} when it is a kernel pair equivalence relation, i.e., $R$ is the kernel pair of some morphism $f$, denoted $R=\Eq(f)$.
\end{itemize}

\subsection{Regular categories}
Recall the definition and main properties of regular categories \cite{BarrGrilletOsdol1971} which will be used throughout this work.
\begin{definition}
\label{definition: regular category}
A category $\C$ is called \emph{regular} if:
\begin{enumerate}
	\item[(i)] $\C$ has finite limits and coequalisers of kernel pairs. 
	\item[(ii)] The class of regular epimorphisms in $\C$ is pullback stable, i.e., if the diagram
	\[
	\xymatrix{
		\bullet \ar[r] \ar@{>>}[d]_p \pb & \bullet \ar@{>>}[d]^f\\
		\bullet \ar[r] & \bullet 
	}
	\]
	is a pullback in $\C$, and $f$ is a regular epimorphism, then so is $p$. 
\end{enumerate}
\end{definition}
The following are important consequences of Definition~\ref{definition: regular category}, and will be used without mention in what follows:
\begin{itemize}
	\item every morphism $f:X \rightarrow Y$ in $\C$ factors as $f = me$ where $e:X \repi  Q$ is a regular epimorphism and $m:Q \mono Y$ a monomorphism. The factorisation $f = me$ is sometimes called an \emph{image factorisation} of $f$. Note that such image factorisations are unique up to isomorphism;
        \item the composite $g f$ of two regular epimorphisms $f: X \repi Y$ and $g:Y \repi Z$ in $\C$ is a regular epimorphism;
	\item if the composite $g f$ of two morphisms $f: X \rightarrow Y$ and $g:Y \rightarrow Z$ in $\C$ is a regular epimorphism, then $g$ is a regular epimorphism.  
\end{itemize}

Given a regular category $\C$, it is possible to define the composition of relations in $\C$. Indeed, let $R$ be a relation between objects $X$ and $Y$, and $S$ a relation between objects $Y$ and $Z$ in a regular category $\C$. Suppose that $r = (r_1,r_2):R \mono X \times Y$ and $s = (s_1,s_2):S \mono Y \times Z$. Consider the diagram:
\[
\xymatrix{
	& &P \pbdown \ar[dl]_{p_1} \ar[dr]^{p_2}  & & \\
	& R \ar[dr]^{r_2} \ar[dl]_{r_1} & & S \ar[dl]_{s_1}\ar[dr]^{s_2} & \\
	X & & Y & & Z
}
\]
where $(P, p_1, p_2)$ is a pullback of $s_1$ along $r_2$. The composite $R\circ S$ is the relation represented by the monomorphism part of any image factorisation of $(r_1p_1, s_2p_2): P \rightarrow X \times Z$ as in the diagram:
\[
\xymatrix{
	P \ar@/_1pc/[rr]_-{(r_1p_1, s_2p_2)} \ar@{->>}[r]^-e &  (R\circ S) \ar@{ >->}[r]^-{r\circ s} & X \times Z.
}
\]
It is well-known that the composition of relations in a regular category is associative (see \cite{BarrGrilletOsdol1971,Borceux1994b}).

\begin{proposition}\emph{\cite{CarboniKellyPedicchio1993}} \label{proposition: relation-composition}
Let $R$ be a relation from $X$ to $Y$ and $S$ a relation from $Y$ to $Z$. Given any morphisms $x:S \to X$ and $z:S\to Z$ then $x(R \circ S)z$ if and only if there exists a regular epimorphism $e: E\repi S$ and a morphism $y:E \rightarrow Y$ such that $xe R y$ and $y S ze$. 
\end{proposition}

From the proposition above it is easy to see that a relation $R$ on an object $A$ is transitive if and only if $R \circ R \leqslant R$.

\subsection{$n$-permutability}\label{n-perm}
Let $\Eq(A)$ denote the meet semilattice of equivalence relations on an object $A$ in a regular category $\C$. Given $R,S\in \Eq(A)$, we follow the notation introduced in~\cite{CarboniKellyPedicchio1993}
\[
    (R,S)_n = \underbrace{R\circ S\circ R\cdots}_n\;;
\]
in particular, $(R,S)_0=\Delta_A$ (the identity relation on $A$), $(R,S)_1=R$, $(R,S)_2=R\circ S$, $(R,S)_3=R\circ S\circ R$, and so on. The relation $(R,S)_n$, $n\geqslant 2$, is not an equivalence relation in general (although it is always reflexive and it is symmetric when $n$ is odd). It is known from~\cite{CarboniKellyPedicchio1993} that $(R,S)_n$ is an equivalence relation, for some $n\geqslant 2$, precisely when the category $\C$ is $n$-permutable.

\begin{definition}
A regular category $\C$ is called \emph{$n$-permutable}, where $n\geqslant 2$, when the equivalence relations in $\C$ are $n$-permutable, that is, given any $R,S\in\Eq(A)$, for any object $A$ of $\C$, we have $(R,S)_n=(S,R)_n$.    
\end{definition}

A $2$-permutable category is usually called a regular \emph{Mal'tsev category}~\cite{CarboniLambekPedicchio1991} and a $3$-per\-mu\-ta\-ble category is usually called a \emph{Goursat category}~\cite{CarboniKellyPedicchio1993} (a Goursat category is regular by definition). It is easy to check that an $n$-permutable category is always $k$-permutable, for any $k\geqslant n$; the converse is false (see~\cite{Schmidt1972}).

When $\C$ is $n$-permutable, $\C$ admits joins of equivalence relations on $A$: for $R,S\in \Eq(A)$
\begin{equation}
\label{join in n-permutable cat}
    R\vee S= (R,S)_n;
\end{equation}
consequently, $\Eq(A)$ is a lattice, for any object $A$ of $\C$. There are several other nice properties of (equivalence) relations in $n$-permutable categories (see~\cite{CarboniKellyPedicchio1993},~\cite{JanelidzeRodeloLinden2014} and~\cite{MartinsRodeloLinden2014}, for example).

\section{Main results}\label{section: trapezoid lemma}
We begin this section by exploring the link between equivalence distributivity and the trapezoid lemma.

Recall from the Introduction that a variety of universal algebras $\V$ satisfies the trapezoid lemma if given congruences $R,S,T$ on any algebra $A$ in $\V$, the diagrammatic condition \eqref{trapezoid lemma} holds. Categorically, we have the following:

\begin{definition} \label{definition: trapezoid lemma}
A category $\C$ is said to satisfy the \emph{trapezoid lemma} when for any three equivalence relations $R,S,T$ on any object $A$ in $\C$ with $R \wedge S \leqslant T$, if morphisms $x,u,v,y$ are such that $xSu$, $u T v$, $y S v$, $xR y$, then we can deduce that $x T y$.   
\end{definition}

An alternative formulation of the trapezoid lemma, which makes use of the context of a regular category $\C$, is simply the following requirement on the equivalence relations $R,S,T$ on any object $A$ in $\C$: 
\begin{equation}\label{TpL}
    R \wedge S \leqslant T\;\;\Rightarrow\;\; R \wedge (S\circ T\circ S) \leqslant T.
\end{equation}
This follows easily from Definition~\ref{definition: trapezoid lemma} and Proposition~\ref{proposition: relation-composition}.

Recall also from the Introduction, that an equivalence distributive category is a finitely complete category in which the lattice of equivalence relations $\Eq(A)$ on any object $A$ forms a distributive lattice (see \cite{GranRodeloNguefeu2020}). 

\begin{proposition}\label{equiv distributive => trapezoid lemma}
Let $\C$ be a regular equivalence distributive category. Then the trapezoid lemma holds in $\C$.
\end{proposition}
\begin{proof}~~
Given any equivalence relations $R,S,T$ on the object $A$ such that $R\wedge S\le T$, we have
\[
R\wedge (S\circ T\circ S) \le R \wedge (S \vee T) \stackrel{\eqref{distributivity}}{=} (R \wedge S) \vee (R \wedge T) \leqslant T,
\]
which gives \eqref{TpL}.
\end{proof}

Regular Mal'tsev or Goursat categories satisfying the categorical versions of the trapezoid lemma and of the triangular lemma \cite{Duda2000} were studied in \cite{GranRodeloNguefeu2020}. There the authors explored the link between these lemmas and equivalence distributivity in the context of regular Mal'tsev or Goursat categories. In what follows, we will establish some basic properties of categories satisfying the trapezoid lemma.  

Let $\C$ be a regular category. Recall from~\cite{FreydScendrov1990} \emph{Freyd's modular law}: given relations $R\mono{A\times B}$, $S\mono B\times C$ and $T\mono{A\times C}$ we have
\begin{equation}\label{Freyd}
	(R\circ S)\wedge T \le (R\wedge (T\circ S^{\op}))\circ S.
\end{equation}

\begin{proposition} \label{proposition: trapezoid relations}
Let $\C$ be a regular category satisfying the trapezoid lemma. Given any equivalence relations $R,S,T$ on any object $A$ in $\C$ such that $S \circ T \leqslant R \circ S$ and $R\wedge S \leqslant T$, then $T$ permutes with $S$. 
\end{proposition}
\begin{proof}
Given the relations as in the statement, we have
\[ 
\begin{array}{lll}
 S\circ T &\leqslant (R\circ S) \wedge (S \circ T)  & \text{(assumption)}\\ 
  & \stackrel{\eqref{Freyd}}{\leqslant} (R\wedge (S\circ T \circ S^{\op}))\circ S \;\;\;\;&  \\
  &= (R\wedge (S\circ T \circ S))\circ S & \text{(S is symmetric)}\\
  & \stackrel{\eqref{TpL}}{\leqslant} T \circ S. 
\end{array}
\]
\end{proof}

Recall that a regular category $\C$ is called \emph{factor permutable} \cite{Gran2004} if for any equivalence relation $E$ on an object $A$ we have $E \circ \Eq(p) = \Eq(p) \circ E$ where $p:A \rightarrow X$ is a product projection of $A$. 

\begin{proposition} \label{proposition: trapezoid implies factor permutable}
A regular category $\C$ satisfying the trapezoid lemma is factor permutable. 
\end{proposition}
\begin{proof}
Given any product diagram 
\[
X \xleftarrow{p_1} A \xrightarrow{p_2} Y
\]
we note that $\Eq(p_1) \wedge \Eq(p_2) = \Delta_A$ and $\Eq(p_1) \circ \Eq(p_2) =\nabla_A$, which denotes the largest equivalence relation on $A$. Now given any equivalence relation $E$ on $A$ and applying Proposition~\ref{proposition: trapezoid relations} in the case $R = \Eq(p_1)$ and $S = \Eq(p_2)$ and $T = E$, we get the desired result.
\end{proof}

\subsection{J\'onsson categories}
In this subsection we introduce the notion of a J\'onsson category of order $n$ as a categorical generalisation of varieties admitting $n$ J\'onsson terms for $n \geqslant 1$. A usual procedure to obtain such kind of generalisation is to consider a characteristic property on relations of the variety being studied, when such a property can be expressed categorically. We note here that a general class of properties which allow for the translation of Mal'tsev conditions into their corresponding categorical conditions has been given in \cite{ZJanelidze2008}. In particular, it was shown in that paper how the \emph{linear Mal'tsev conditions} (see \cite{Snow1999}) can be interpreted as a property determined by an extended matrix of variables.

\begin{definition}\label{definition: Jonsson Category} A regular category $\C$ is said to be a \emph{J\'onsson category of order $n$} (where $n \in \mathbb{N})$ if for any equivalence relations $R,S,T$ on any object $A$ of $\C$, we have
\begin{equation}\label{Jonsson pp}
	R\wedge (S\circ T)\leqslant ( R\wedge S, R\wedge T)_{n+1}.
\end{equation}
\end{definition}
\begin{remark}
A variety of universal algebras $\V$ admits $n$ J\'onsson terms if and only if it is a J\'onsson category of order $n$. The proof of this is implicitly in \cite{Jonsson1967}, and for this reason we omit it. 
\end{remark}
\begin{remark} \label{remark: majority categories}
A J\'onsson category of order 1 is precisely a regular majority category (see Theorem 3.1 in~\cite{Hoefnagel2020a}). Note also, that the notion of majority category~\cite{Hoefnagel2019a} does not require the base category to be regular and has been defined as a property which may be formulated with respect to any category. 
\end{remark}

\begin{theorem} \label{theorem: main theorem}
Let $\C$ be a regular category. The following statements are equivalent for $n\geqslant 1$:
\begin{enumerate}
    \item[\emph{(i)}] $\C$ is a J\'onsson category of order $n$;
    \item[\emph{(ii)}] Given an equivalence relation $R$ and reflexive relations $S,T$ on any object $A$ of $\C$, we have
    \begin{equation}\label{Jonsson pp for vars}
    R\wedge (S\circ T)\leqslant \left(\, (R\wedge S^\op)\circ (R\wedge S),(R\wedge T)\circ(R\wedge T^\op) \,\right)_{n+1}.
    \end{equation}
\end{enumerate}
\end{theorem}
The property (ii) in the theorem above has been extracted from \cite{Jonsson1967} (see equation (4) therein). The proof of the implication (ii) $\Rightarrow$ (i) is straightforward by using the symmetry and transitivity of $S$ and $T$. Moreover, for $n=1$, the implication (i) $\Rightarrow$ (ii) follow from Theorem 3.1 in~\cite{Hoefnagel2020a}. The proof of the implication (i) $\Rightarrow$ (ii) for the general case when $n\geqslant 2$ is a bit technical. For this reason, we decided to postpone its proof to the end of this work (Section~\ref{section: proof of main thm}), while focussing now on the nice consequences and properties that can be extracted from Theorem~\ref{theorem: main theorem}.

\begin{corollary}\label{corollary: J_n implies J_k}
If $\C$ is a J\'onsson category of order $n$, then it is a J\'onsson category of order $k$, for any $k\geqslant n$.
\end{corollary}
\begin{proof}
    This follows immediately from Definition~\ref{definition: Jonsson Category} and the fact that $(R\wedge S,R\wedge T)_{n+1}\leqslant (R\wedge S,R\wedge T)_{k+1}$, since $R,S$ and $T$ are reflexive relations.
\end{proof}

Under the assumption that the base category $\C$ is $n$-permutable (so that suprema of equivalence relations on a same object $A$ exist in $\C$, so that $\Eq(A)$ is a lattice), we shall prove in Proposition~\ref{proposition: n-permutable and J_n} that the notion of equivalence distributive category coincides with that of a J\'onsson category (of order $n-1$).

\begin{proposition}
\label{proposition: n-permutable+Jonsson implies ...}
Let $\C$ be an $n$-permutable category ($n\geqslant 2$) and a J\'onsson category of any order. Given $R,S,T\in \Eq(A)$, for some object $A$ of $\C$, we have
\[
    R\wedge (S,T)_j \leqslant (R\wedge S)\vee (R\wedge T),\; \text{for any $j\geqslant 1$}.
\]
\end{proposition}
\begin{proof}
Suppose that $\C$ is a J\'onsson category of order $k\geqslant 1$. We can suppose that $k\geqslant n-1$ by Corollary~\ref{corollary: J_n implies J_k}. Since $k\geqslant n-1$, then $\C$ is also $(k+1)$-permutable. For $j=1$, the inequality is obvious. We prove the above inequality by induction on $j\geqslant 2$. For $j=2$, we have
$$
\begin{array}{lll}
    R\wedge (S,T)_2 & =  R\wedge (S\circ T)& \\
    &\leqslant  (\, R\wedge S, R\wedge T\,)_{k+1} \;\;\;\; & \text{(by Definition~\ref{definition: Jonsson Category})} \\
    & =  (R\wedge S)\vee (R\wedge T) &\text{($\,(k+1)$-permutability)}.
\end{array}
$$
Suppose now that $R\wedge (S,T)_j\leqslant (R\wedge S)\vee (R\wedge T)$, for some $j\geqslant 2$. When $j$ is odd
$$
\begin{array}{ll}
    R\wedge (S,T)_{j+1} = R\wedge ( (S,T)_j \circ T)\;\;\;\;\;\; & \\
    \stackrel{\eqref{Jonsson pp for vars}}{\leqslant}  (\, (R\wedge (S,T)_j)\circ (R\wedge (S,T)_j), R\wedge T\,)_{k+1} & \text{($(S,T)_j^\op=(S,T)_j$, $T^\op=T$)} \\
    \leqslant (\, ((R\wedge S) \vee (R\wedge T))\circ ((R\wedge S) \vee (R\wedge T)) , R\wedge T\,)_{k+1} & \text{(induction assumption)}\\
    = (\, (R\wedge S) \vee (R\wedge T), R\wedge T\,)_{k+1} & \text{($(R\!\wedge\! S)\!\vee\! (R\!\wedge\! T)$ transitive)}\\
    = ( (R\wedge S)\vee (R\wedge T))\vee (R\wedge T) &\text{($\,(k+1)$-permutability)}\\
    = (R\wedge S)\vee (R\wedge T).
\end{array}
$$
When $j$ is even
$$
\begin{array}{ll}
    R\wedge (S,T)_{j+1}  = R\wedge ( (S,T)_j\circ S)\\
    \stackrel{\eqref{Jonsson pp for vars}}{\leqslant} (\,(R\wedge(T,S)_j)\circ ( R\wedge (S,T)_j), R\wedge S\,)_{k+1},& \text{($(S,T)_j^\op=(T,S)_j$, $T^\op=T$)} \\
    \leqslant (\, ((R\wedge S) \vee (R\wedge T))\circ ((R\wedge S) \vee (R\wedge T)), R\wedge S\,)_{k+1} & \text{(induction assumption)}\\
    = (\, (R\wedge S) \vee (R\wedge T), R\wedge S\,)_{k+1} & \text{($(R\!\wedge\! S)\!\vee\! (R\!\wedge\! T)$ transitive)}\\
    = ( (R\wedge S)\vee (R\wedge T))\vee (R\wedge S) & \text{($\,(k+1)$-permutability)}\\
    = (R\wedge S)\vee (R\wedge T).
\end{array}
$$
\end{proof}

\begin{proposition} \label{proposition: n-permutable and J_n}
Let $\C$ be an $n$-permutable category ($n\geqslant 2$). Then $\C$ is equivalence distributive if and only $\C$ is a J\'onsson category (of order $n-1$).
\end{proposition}
\begin{proof}
Note that for $n=2$, the statement claims the known fact: a regular Mal'tsev category is equivalence distributive if and only if it is a majority category (Corollary 3.2 in~\cite{Hoefnagel2020a}).

Let $R,S,T\in \Eq(A)$, for some object $A$ of $\C$. Suppose that $\C$ is a J\'onsson category (of any order, in this implication). By Proposition~\ref{proposition: n-permutable+Jonsson implies ...} we get
\[
    R\wedge (S\vee T) \stackrel{\eqref{join in n-permutable cat}}{=} R\wedge (S,T)_n \leqslant (R\wedge S)\vee (R\wedge T).
\]

Conversely, if $\C$ is equivalence distributive, then
\[
R\wedge (S\circ T) \leqslant R \wedge (S,T)_n \stackrel{\eqref{join in n-permutable cat}}{=} R\wedge (S\vee T) \stackrel{\eqref{distributivity}}{=} (R \wedge S)\vee (R\wedge T) \stackrel{\eqref{join in n-permutable cat}}{=} (R\wedge S,R\wedge T)_n.
\]
By Definition~\ref{definition: Jonsson Category}, $\C$ is a J\'onsson category of order $n-1$.
\end{proof}


\begin{proposition}
\label{J_n-closed implies trapezoid lemma}
If $\C$ is a J\'onsson category of any order, then $\C$ satisfies the trapezoid lemma. 
\end{proposition}

\begin{proof}
Suppose that $\C$ is a J\'onsson category of order $n$. Let $R,S,T\in\Eq(A)$, for some object $A$ of $\C$, such that $R \wedge S \leqslant T$. We may apply \eqref{Jonsson pp} to get that
\[
R\wedge (S\circ T) \leqslant (R \wedge S, R \wedge T)_{n + 1} \leqslant T,
\]
the last inequality following from the property that both $R \wedge S$ and $R \wedge T$ are contained in $T$. Similarly, we have that $R \wedge (T\circ S) \leqslant T$. We prove \eqref{TpL} 
\[
\begin{array}{ll}
    R \wedge ((S\circ T)\circ S)  \\
    \leqslant  ((R \wedge (T\circ S))\circ (R \wedge (S\circ T)), R \wedge S)_{n+1}\;\;\;\;\;\; & \text{(by \eqref{Jonsson pp for vars}, $(S\circ T)^\op=T\circ S$, $S^\op=S$)}\\
    \leqslant T.\\ 
\end{array}
\]
\end{proof}

Recall from~\cite{Gumm1983} that a variety of universal algebras $\V$ satisfies the \emph{shifting lemma} if for every three congruences $R,S,T$ on any algebra $A$ in $\V$ the a diagrammatic condition similar to that of \eqref{trapezoid lemma} holds, where $uTv$ is replaced with $u(R\wedge T)v$. Finitely complete categories satisfying (the categorical version of) the shifting lemma are called \emph{Gumm} categories, and were introduced in \cite{BournGran2004}. Any regular Mal'tsev or Goursat category satisfy the shifting lemma~\cite{CarboniLambekPedicchio1991, CarboniKellyPedicchio1993}, so they are examples of Gumm categories. Furthermore, it is shown in~\cite{GranRodeloNguefeu2019} that regular Mal'tsev or Goursat categories can be characterised by stronger variations of the shifting lemma. The shifting lemma is  implied by the trapezoid lemma (even for finitely complete categories), so that every J\'onsson category $\C$ is necessarily a Gumm category (by Proposition~\ref{J_n-closed implies trapezoid lemma}). It then follows from Proposition~2.12 in \cite{Hoefnagel2019b} that binary products commute with coequalisers locally (in the sense of \cite{Hoefnagel2019b}) in $\C$, and that $\C$ is locally anticommutative in the sense of \cite{Hoefnagel2021} (see Corollary~2.18 and Remark~2.17 of \cite{Hoefnagel2021}). 

\begin{remark}\label{Gumm not imply Jonsson}
There are examples of Gumm categories that are not J\'onsson categories. Indeed, any regular Mal'tsev category which is not equivalence distributive, is not a J\'onsson category of any order. This follows from the proof of Proposition~\ref{proposition: n-permutable and J_n}. As examples, we have the (quasi)varieties $\mathsf{Ab}$ of abelian groups, $\Grp$ of groups or $\mathsf{Rng}$ of rings.
\end{remark}

\begin{definition}
Let $A$ be an object in a category $\C$, then a \emph{factor relation} $F$ on $A$ is any equivalence relation on $A$ of the form $\Eq(p)$ where $X \xleftarrow{p} A \xrightarrow{p'} Y$ is a product diagram. The \emph{factor relation} $F' = \Eq(p')$ will be called a \emph{complementary} factor relation of $F$.
\end{definition}

\begin{proposition} \label{proposition: factor-relation formula}
For any J\'onsson category $\C$ of any order, given $F,S,T \in \Eq(A)$ where $F$ is a factor relation, we have 
\[
(F\circ S) \wedge (F\circ T) = F \circ(S \wedge T).
\]
\end{proposition}
\begin{proof}
Suppose that $\C$ is J\'onsson of order $n$ for some $n \geqslant 1$. Note that since $\C$ is factor permutable, the relation $F \circ S$ is an equivalence relation, so that we have
\[\begin{array}{ll}
	(F\circ S) \wedge (F\circ T) \\
	\leqslant ((F\circ S) \wedge F, (F\circ S) \wedge T)_{n+1} \;\;\;\;\;\; & \text{(by \eqref{Jonsson pp})}\\ 
	\leqslant (F, (F\wedge T, S \wedge T)_{n+1})_{n+1} & \text{($(F\circ S)\wedge F \leqslant F$, \eqref{Jonsson pp} applied to $(F\circ 	S)\wedge T$)}\\
	\leqslant (F, (F, S \wedge T)_{n+1})_{n+1} & \text{($F\wedge T\leqslant F$)}\\
	= F\circ (S \wedge T), & \text{($F\circ (S\wedge T)=(S\wedge T)\circ F$ and $F$ is transitive)}\\
\end{array}
\]
The inclusion $F \circ (S \wedge T) \leqslant (F\circ S) \wedge (F \circ T)$ is trivial. 
\end{proof}

\begin{proposition}
If $\C$ is a Barr-exact J\'onsson category of any order, then the set of factor relations $\F(A)$, for any object $A$ of $\C$, forms a Boolean algebra under $\wedge$ and $\circ$. 
\end{proposition}
\begin{proof}
Let $\C$ be J\'onsson of order $n\geqslant 1$ and suppose that $F, G \in \F(A)$. Then we show that $F\wedge G$ and $F'\circ G'$ are complementary factor relations. By \eqref{Jonsson pp}, we have
\[
(F\wedge G) \wedge (F' \circ G') \leqslant (F \wedge G \wedge F', F \wedge G \wedge G')_{n+1} \leqslant (F\wedge F', G\wedge G')_{n+1}= \Delta_A,
\]
since $F\wedge F'=\Delta_A=G\wedge G'$, as complementary factor relations.
Moreover, applying Proposition~\ref{proposition: factor-relation formula} we have
\[
	(F \wedge G) \circ (F' \circ G') 
	= ((F\wedge G)\circ F') \circ G'
	= ((F \circ F') \wedge (G \circ F'))  \circ G'
	= G \circ F'  \circ G'  =
	\nabla_A,
\]
since $F\circ F'=\nabla_A=G\circ G'$, as complementary factor relations. Therefore we have that $F\wedge G$ and $F'\circ G'$ complementary factor relations, so that $\F(A)$ is closed under $\wedge$ and $\circ$, and then by Proposition~\ref{proposition: factor-relation formula} it is a complemented distributive lattice under these terms. 
\end{proof}
\begin{remark}
 The above proposition shows that in any Barr-exact J\'onsson category, every object with global support is \emph{projection coextensive} in the sense of \cite{Hoefnagel2020b}.   
\end{remark} 

\subsection{Examples of J\'onsson categories}

Any regular majority category is a J\'onsson category of order 1 (see~\cite{Hoefnagel2019a} for details). We also know that a regular Mal'tsev (=$2$-permutable) category is a majority category if and only if it is equivalence distributive (Proposition~\ref{proposition: n-permutable and J_n} for $n=2$). We have the following list of examples, where we also add whether the category is a Mal'tsev category or an equivalence distributive category, or not, when such is known. Note that, in the varietal context equivalence distributive is usually called \emph{congruence distributive}. Also, a Mal'tsev and congruence distributive variety is called  an \emph{arithmetical variety} and a Barr-exact Mal'tsev equivalence distributive category is called an \emph{arithmetical category}~\cite{Bourn2001} (dropping the existence of coequalisers given in the original definition~\cite{Pedicchio1996}).

\noindent Varietal examples: 
\begin{itemize}
	\item $\mathsf{Lat}$, the variety of lattices (it is not a Mal'tsev category, but it is congruence distributive);
	\item $\mathsf{lGrp}$, the variety of lattice-ordered groups (it is an arithmetical variety~\cite{Birkhoff1942});
	\item $\mathsf{Heyt}$, the variety of Heyting algebras (it is an arithmetical variety);
	\item $\mathsf{Bool}$, the variety of Boolean algebras (it is an arithmetical variety);
	\item any Mal'tsev variety equipped with a $k$-ary near unanimity term, $k\geqslant 1$ (it is an arithmetical variety~\cite{Mitchke1978}). 
\end{itemize}

\noindent Non-varietal examples: 
\begin{itemize}
	\item any quasivariety of those listed above;
	\item the variety of distributive lattices equipped with an additional term of countable arity (it is not Mal'tsev nor an equivalence distributive category~\cite{GJanelidze2016});
	\item $\mathsf{((pre)topos)^{op}}$, the dual of any (pre)topos (it is an arithmetical category~\cite{CarboniKellyPedicchio1993, Pedicchio1995});
	\item $\mathsf{Top^{op}}$, the dual of the category of topological spaces (it is not a Mal'tsev category~\cite{Weighill2017});
	\item $\mathsf{(Met_{\infty})^{\op}}$, the dual of the category of (extended) metric spaces (it is not a Mal'tsev category~\cite{Weighill2017});
	\item $\mathsf{NReg}$, the category of von Neumann regular rings (it is an arithmetical category~\cite{Hoefnagel2019a});
	\item $\mathsf{NReg(Top)}$, the category of topological von Neumann regular rings (it is a regular Mal'tsev and congruence distributive category~\cite{Hoefnagel2020a});
	\item $\mathsf{Bool(Top)}$, the category of topological Boolean rings (it is a regular Mal'tsev and congruence distributive category~\cite{Hoefnagel2020a});
	\item  the category of topological lattices~\cite{Hoefnagel2019a}.
\end{itemize}

A Goursat (i.e. $3$-permutable) category is a J\'onsson category of order 2 (but not of order 1) if and only if it is equivalence distributive (Proposition~\ref{proposition: n-permutable and J_n} for $n=3$). For example, the variety $\textbf{Imp}$ of implication algebras~\cite{Mitschke1971} since it is known that such a variety is $3$-permutable (not $2$-permutable) and congruence distributive. As shown in \cite{Mitchke1978} every variety which admits a $k$-ary near unanimity term~\cite{Mitschke1978} will be an example of a J\'onsson category of order $2k-5$ (see~\cite{Mitchke1978}).

Finally, if $\C$ is a J\'onsson category of order $n$ and $X$ is an object in $\C$, then so are the comma categories $\C/X$, $X/\C$, the category $\mathsf{Pt}_X(\C)$ of points over $X$~\cite{Bourn1996} or the functor category $\C^{\D}$, for any category $\D$ (see~\cite{Hoefnagel2019a} for details when $n=1$; similar arguments hold for $n\geqslant 2$).


\subsection{Proof of Theorem~\ref{theorem: main theorem}}\label{section: proof of main thm}
We prove the implication (i) $\Rightarrow$ (ii) of Theorem~\ref{theorem: main theorem} when $n\geqslant 2$.

Let $a,c\colon X\to A$ be morphisms such that $(a,c)\in_X R\wedge (S\circ T)$. From $(a,c)\in_X S\circ T$ there exists a regular epimorphism $e_1\colon E_1\repi X$ and a morphism $b\colon E_1\to A$ such that $(ae_1,b)\in_{E_1} S$ and $(b,ce_1)\in_{E_1} T$ and also $(ae_1,ce_1)\in_{E_1} R$. We define a ternary relation $D$ on $A$ as the image $\rho(W)$, where $W$ is the quaternary relation on $A$ represented by $w$ as in the diagram
$$
\xymatrix@C=80pt@R=30pt{
    D \ar@{ >->}[d]_-d &  & W \ar[r] \ar@{ >->}[d] \ar@{ >->}[dl]_-w \pb \ar@{>>}[ll]_-p & S\times R_0\times T \ar@{ >->}[d]^-{s\times r\times (t_2,t_1)} \\
    A^3 & A^4 \ar[l]^-{\rho=(\pi_1,\pi_3,\pi_2)} & R\times A^2 \ar[r]_-{\pi=(r_1\pi_1, \pi_3, \pi_2, \pi_3, r_2\pi_1, \pi_4)} \ar@{ >->}[l]^-{r\times 1_{A^2}}  & A^6. }
$$
Given $x,y,z,w\colon B\to A$, we have
$$
\begin{array}{lcl}
    (x,z,y,w)\in_B W & \Leftrightarrow & (x,z)\in_B R, (y,w)\in_B R, (x,w)\in_B S, (w,z)\in_B T,\\
    (x,y,z)\in_B D & \Leftrightarrow & \exists u\colon C \repi B, v\colon C\to A\;\;\mathrm{such\;\;that} \\
    & & (x,z)\in_B R, (yu,v)\in_C R, (xu,v)\in_C S, (v,zu)\in_C T.
\end{array}
$$

We have
\begin{itemize}
     \item $(ae_1,ae_1,ae_1)\in_{E_1} D$, since there exist $1_{E_1}\colon E_1\repi E_1$ and $ae_1\colon E_1\to A$ such that $(ae_1,ae_1)\in_{E_1} R,S,T$; 
     \item $(ae_1,b,ce_1)\in_{E_1} D$, since there exist $1_{E_1}\colon E_1\repi E_1$ and $b\colon E_1\to A$ such that $(ae_1,ce_1)$ $\in_{E_1} R$, $(b,b)\in_{E_1} R$, $(ae_1,b)\in_{E_1} S$, $(b,ce_1)\in_{E_1} T$;
     \item $(ce_1,ae_1,ce_1)\in_{E_1} D$, since there exist $1_{E_1}\colon E_1\repi E_1$ and $ce_1\colon E_1\to A$ such that $(ce_1,ce_1)\in_{E_1} R$, $(ae_1,ce_1)\in_{E_1} R$, $(ce_1,ce_1)\in_{E_1} S$, $(ce_1,ce_1)\in_{E_1} T$.
\end{itemize}

We consider the equivalence relations on $D$ defined by the kernel pairs of the projections $d_i\colon D\to A$: $D_1=\Eq(d_1)$, $D_2=\Eq(d_2)$ and $D_3=\Eq(d_3)$. We can apply our assumption to these equivalence relations
\[
	D_2\wedge (D_1\circ D_3)\leqslant (D_2\wedge D_1,D_2\wedge D_3)_{n+1}.
\]

We have, 
\[\begin{array}{l}
	((ae_1,ae_1,ae_1), (ce_1,ae_1,ce_1))\eot{E_1} D_2, \\
	((ae_1,ae_1,ae_1), (ae_1,b,ce_1))\eot{E_1} D_1 \;\;\text{and} \;\; ((ae_1,b,ce_1), (ce_1,ae_1,ce_1))\eot{E_1} D_3; \\
	\text{it follows that} 
	((ae_1,ae_1,ae_1), (ce_1,ae_1,ce_1))\eot{E_1} D_2\wedge (D_1\circ D_3). 
\end{array}
\]
From the inequality above, we conclude that 
\begin{equation}\label{aux}
	((ae_1,ae_1,ae_1), (ce_1,ae_1,ce_1))\eot{E_1}(D_2\wedge D_1,D_2\wedge D_3)_{n+1}.
\end{equation}
When $n$ is odd, the sequence \eqref{aux} ends with a $D_2\wedge D_3$. So,  there exists a regular epimorphism $e_2\colon E_2\repi E_1$ and morphisms $z_1$, $\cdots$, $z_{n-1}\colon E_2\to A$ such that the relations on the left part of the following table hold. Thus, there exist  $e_3\colon E_3\repi E_2$ and $d_1, \cdots, d_n\colon E_3\to A$ such that the relations on the right part of the following table hold
$$
\begin{array}{|c|l|ll|}
\hline
     1 & (ae_1e_2,ae_1e_2,z_2)\in_{E_2} D & (ae_1e_2,z_2)\in_{E_2} R & (ae_1e_2e_3,d_1)\in_{E_3} R \\
     && (ae_1e_2e_3,d_1)\in_{E_3} S & (d_1,z_2e_3)\in_{E_3} T\\
     2 & (z_1,ae_1e_2,z_2)\in_{E_2} D & (z_1,z_2)\in_{E_2} R & (ae_1e_2e_3,d_2)\in_{E_3} R \\
     && (z_1e_3,d_2)\in_{E_3} S & (d_2,z_2e_3)\in_{E_3} T\\
     3 & (z_1,ae_1e_2,z_4)\in_{E_2} D & (z_1,z_4)\in_{E_2} R & (ae_1e_2e_3,d_3)\in_{E_3} R\\ 
     && (z_1e_3,d_3)\in_{E_3} S & (d_3,z_4e_3)\in_{E_3} T\\
     \vdots & \vdots & \vdots & \vdots \\
     n-2 & (z_{n-4},ae_1e_2,z_{n-1})\in_{E_2} D & (z_{n-4},z_{n-1})\in_{E_2} R & (ae_1e_2e_3,d_{n-2})\in_{E_3} R\\
     && (z_{n-4}e_3,d_{n-2})\in_{E_3} S & (d_{n-2},z_{n-1}e_3)\in_{E_3} T \\
     n-1 & (z_{n-2},ae_1e_2,z_{n-1})\in_{E_2} D & (z_{n-2},z_{n-1})\in_{E_2} R & (ae_1e_2e_3,d_{n-1})\in_{E_3} R\\
     && (z_{n-2}e_3,d_{n-1})\in_{E_3} S & (d_{n-1},z_{n-1}e_3)\in_{E_3} T \\
     n & (z_{n-2},ae_1e_2,ce_1e_2)\in_{E_2} D & (z_{n-2},ce_1e_2)\in_{E_2} R & (ae_1e_2e_3,d_{n})\in_{E_3} R\\ 
     && (z_{n-2}e_3,d_{n})\in_{E_3} S & (d_{n},ce_1e_2e_3)\in_{E_3} T\\
\hline
\end{array}
$$

We may deduce the following relations from the right side of the table
\begin{itemize}
    \item[(a)] $(d_i,d_{i+1})\in_{E_3} R$, $\forall i\in\{1,\cdots, n-1\}$, because $(ae_1e_2e_3,d_i)\in_{E_3} R$, $\forall i\in\{1,\cdots, n\}$ and $R$ is symmetric and transitive;
    \item[(b)] $(d_i,z_{i+1}e_3)\in_{E_3} T$ (line $i$) and $(z_{i+1}e_3,d_{i+1})\in_{E_3} T^\op$ (line $i+1$), $\forall i\in\{1,\cdots, n-2\}$ and $i$ odd;
    \item[(c)] $(d_i,z_{i-1}e_3)\in_{E_3} S^\op$ (line $i$) and $(z_{i-1}e_3,d_{i+1})\in_{E_3} S$ (line $i+1$), $\forall i\in\{2,\cdots, n-1\}$ and $i$ even.
\end{itemize}
Using the symmetry and transitivity of $R$ we also get 
\begin{itemize}
    \item[(d)] $(d_i,z_{i+1}e_3)\in_{E_3} R$ and $(z_{i+1}e_3,d_{i+1})\in_{E_3} R$, $\forall i\in\{1,\cdots, n-2\}$ and $i$ odd. For example, when $i=1$, we get $(d_1,z_2e_3)\in_{E_3} R$ from the relations $(d_1,ae_1e_2e_3)\in_{E_3}R$ and $(ae_1e_2e_3,z_2e_3)\in_{E_3}R$ in the first line of the table above; also $(z_2e_3,d_2)\in_{E_3}R$ follows from $(z_2e_3,d_1)\in_{E_3}R$ (by the previously shown and symmetry of $R$) and $(d_1,d_2)\in_{E_3} R$ from (a);
    \item[(e)] $(d_i,z_{i-1}e_3)\in_{E_3} R$ and $(z_{i-1}e_3,d_{i+1})\in_{E_3} R$, $\forall i\in\{2,\cdots, n-1\}$ and $i$ even. For example, when $i=2$, we get $(d_2,z_1e_3)\in_{E_3} R$ from the relations $(d_2,z_2e_3)\in_{E_3}R$, when $i=1$, and $(z_2e_3,z_1e_3)\in_{E_3}R$ which follows from the second line of the table above; also $(z_1e_3,d_3)\in_{E_3}R$ follows from $(z_1e_3,d_2)\in_{E_3}R$ (by the previously shown and symmetry of $R$) and $(d_2,d_3)\in_{E_3} R$  from (a).
\end{itemize}

Note that (d) is equal to (b) by replacing $T$ with $R$ and (e) is equal to (c) by replacing $S$ with $R$. Combining the above we obtain the following relations (next we write $\cdots W \cdots$ to denote $(\cdots, \cdots)\in_{E_3} W$, for a relation $W$)
$$
\begin{array}{ccccccccc}
    ae_1e_2e_3 & (R\wedge S^\op) & ae_1e_2e_3 & (R\wedge S) & d_1 & (R\wedge T)&  z_2e_3 & (R\wedge T^\op) & d_2 \\
    d_2 & (R\wedge S^\op) & z_1e_3 & (R\wedge s) & 
    d_3 & (R\wedge T) & z_4e_3 & (R\wedge T^\op) & d_4\\
    \vdots \\
    d_{n-3} & (R\wedge S^\op) & z_{n-4}e_3 & (R\wedge S) & d_{n-2} & (R\wedge T) & z_{n-1}e_3 & (R\wedge T^\op) & d_{n-1} \\
    d_{n-1} & (R\wedge S^\op) & z_{n-2}e_3 & (R\wedge S) & d_{n} & (R\wedge T) & ce_1e_2e_3 & (R\wedge T^\op) & ce_1e_2e_3;
\end{array}
$$
so
$$
(ae_1e_2e_3,ce_1e_2e_3)\in_{E_3} ((R\wedge S^\op)\circ (R\wedge S), (R\wedge T)\circ (R\wedge T^\op)\,)_{n+1}.
$$ 
Since $e_1e_2e_3\colon E_3 \repi X$ is a regular epimorphism, we obtain 
$$
(a,c)\in_X ((R\wedge S^\op)\circ (R\wedge S), (R\wedge T)\circ (R\wedge T^\op)\,)_{n+1}.
$$

When $n$ is even ($n\geqslant 2$), the proof is similar except that the sequence in \eqref{aux} ends with a $D_2\wedge D_1$. Now we get the table
$$
\begin{array}{|c|l|ll|}
\hline
     1 & (ae_1e_2,ae_1e_2,z_2)\in_{E_2} D & (ae_1e_2,z_2)\in_{E_2} R & (ae_1e_2e_3,d_1)\in_{E_3} R \\
     && (ae_1e_2e_3,d_1)\in_{E_3} S & (d_1,z_2e_3)\in_{E_3} T\\
     2 & (z_1,ae_1e_2,z_2)\in_{E_2} D & (z_1,z_2)\in_{E_2} R & (ae_1e_2e_3,d_2)\in_{E_3} R \\
     && (z_1e_3,d_2)\in_{E_3} S & (d_2,z_2e_3)\in_{E_3} T\\
     3 & (z_1,ae_1e_2,z_4)\in_{E_2} D & (z_1,z_4)\in_{E_2} R & (ae_1e_2e_3,d_3)\in_{E_3} R\\ 
     && (z_1e_3,d_3)\in_{E_3} S & (d_3,z_4e_3)\in_{E_3} T\\
     \vdots & \vdots & \vdots & \vdots \\
     n-2 & (z_{n-3},ae_1e_2,z_{n-2})\in_{E_2} D & (z_{n-3},z_{n-2})\in_{E_2} R & (ae_1e_2e_3,d_{n-2})\in_{E_3} R\\
     && (z_{n-3}e_3,d_{n-2})\in_{E_3} S & (d_{n-2},z_{n-2}e_3)\in_{E_3} T \\
     n-1 & (z_{n-3},ae_1e_2,z_{n})\in_{E_2} D & (z_{n-3},z_{n})\in_{E_2} R & (ae_1e_2e_3,d_{n-1})\in_{E_3} R\\
     && (z_{n-3}e_3,d_{n-1})\in_{E_3} S & (d_{n-1},z_{n}e_3)\in_{E_3} T \\
     n & (ce_1e_2,ae_1e_2,z_n)\in_{E_2} D & (ce_1e_2,z_n)\in_{E_2} R & (ae_1e_2e_3,d_{n})\in_{E_3} R\\ && (ce_1e_2e_3,d_{n})\in_{E_3} S & (d_{n},z_ne_3)\in_{E_3} T\\
\hline
\end{array}
$$

Similarly, we may deduce
\begin{itemize}
    \item $(d_i,d_{i+1})\in_{E_3} R$, $\forall i\in\{1,\cdots, n-1\};$
    \item $(d_i,z_{i+1}e_3)\in_{E_3} (R\wedge T)$ and $(z_{i+1}e_3,d_{i+1})\in_{E_3} (R\wedge T^\op)$, $\forall i\in\{1,\cdots, n-1\}$ and $i$ odd;
    \item $(d_i,z_{i-1}e_3)\in_{E_3} (R\wedge S^\op)$ and $(z_{i-1}e_3,d_{i+1})\in_{E_3} (R\wedge S)$, $\forall i\in\{2,\cdots, n-2\}$ and $i$ even (these do not hold when $n=2$).
\end{itemize}

Combining the above we obtain the following relations
$$
\begin{array}{ccccccccc}
    ae_1e_2e_3 & (R\wedge S^\op) & ae_1e_2e_3 & (R\wedge S) & d_1 & (R\wedge T) & z_2e_3 & (R\wedge T^\op) & d_2 \\
    d_2 & (R\wedge S^\op) & z_1e_3 & (R\wedge s) & 
    d_3 & (R\wedge T) & z_4e_3 & (R\wedge T^\op) & d_4\\
    \vdots \\
    d_{n-2} & (R\wedge S^\op) & z_{n-3}e_3 & (R\wedge S) & d_{n-1} & (R\wedge T) & z_{n}e_3 & (R\wedge T^\op) & d_{n} \\
    d_{n} & (R\wedge S^\op) & ce_1e_2e_3 & (R\wedge S) & ce_1e_2e_3;
\end{array}
$$
so
$$
(ae_1e_2e_3,ce_1e_2e_3)\in_{E_3} ((R\wedge S^\op)\circ (R\wedge S), (R\wedge T)\circ (R\wedge T^\op)\,)_{n+1}.
$$ 
Since $e_1e_2e_3\colon E_3 \repi X$ is a regular epimorphism, we obtain 
$$
(a,c)\in_X ((R\wedge S^\op)\circ (R\wedge S), (R\wedge T)\circ (R\wedge T^\op)\,)_{n+1}.
$$
This completes the proof.

\begin{remark}
Part of the proof of Theorem~\ref{theorem: main theorem} follows the procedure introduced in~\cite{JanelidzeRodeloLinden2014}. In~\cite{JanelidzeRodeloLinden2014} the authors translate varietal proofs into categorical ones by using an appropriate matrix property corresponding to the ground categorical context.
\end{remark}

\bibliographystyle{plain}

\end{document}